\documentclass{amsart}
\usepackage[all]{xy}

\newcounter{zlist}
\newenvironment{zlist}{\begin{list}{{\rm(\arabic{zlist})}}{
\usecounter{zlist}\leftmargin2.5em\labelwidth2em\labelsep0.5em
\topsep0.6ex\itemsep0.3ex plus0.2ex minus0.3ex
\parsep0.3ex plus0.2ex minus0.1ex}}{\end{list}}

\newcounter{blist}

\newcounter{rlist}
\newenvironment{rlist}{\begin{list}{{\rm(\roman{rlist})}}{
\usecounter{rlist}\leftmargin2.5em\labelwidth2em\labelsep0.5em
\topsep0.6ex\itemsep0.3ex plus0.2ex minus0.3ex
\parsep0.3ex plus0.2ex minus0.1ex}}{\end{list}}

\swapnumbers
\newtheorem{theorem}{Theorem}[section]
\newtheorem{lemma}[theorem]{Lemma}
\newtheorem{thm}[theorem]{}
\newtheorem{proposition}[theorem]{Proposition}

\newtheorem{corollary}[theorem]{Corollary}
\newtheorem{remark}[theorem]{Remark}
\numberwithin{equation}{section}

\newcommand{\bG}{{\mathcal{G}}}
\newcommand{\bT}{{\mathcal{T}}}
\newcommand{\bS}{{\mathcal{S}}}
\newcommand{\bZ}{{\mathcal{Z}}}
\newcommand{\D}{{\mathbb{D}}}
\newcommand{\A}{{\mathbb{A}}}

\newcommand{\B}{{\mathbb{B}}}
\newcommand{\M}{{\mathbb{M}}}
\newcommand{\V}{{\mathbb{V}}}
\newcommand{\wG}{{\widehat{G}}}
\newcommand{\bbT}{{\bT^l_{A}}}
\newcommand{\bbG}{{\bG^l_{\mathcal{Z}}}}
\newcommand{\ve}{\varepsilon}
\newcommand{\xr}{\xrightarrow}
\newcommand{\ot}{\otimes}

\begin{document}

\title{Galois functors and generalised Hopf modules}
 \author{Bachuki Mesablishvili and Robert Wisbauer}

\begin{abstract}
As shown in a previous paper by the same authors, the theory of Galois functors provides a categorical
framework for the characterisation of bimonads on any category as {\em Hopf monads} and also for
 the characterisation of opmonoidal monads on monoidal categories
as {\em right Hopf monads} in the sense of Brugui\`eres and Virelizier.
Hereby the central part is to describe conditions under which a comparison functor between the base
category and the category of Hopf modules becomes an equivalence ({\em Fundamental Theorem}).

For monoidal categories,  Aguiar and Chase
extended the setting by replacing the base category by a comodule category for some comonoid and considering a comparison functor to {\em generalised Hopf modules}.
For duoidal categories, B\"ohm, Chen and Zhang
investigated a comparison functor to the Hopf modules over a {\em bimonoid}  induced by the two
monoidal structures given in such categories. In both approaches  fundamental theorems are proved and
the purpose of this paper is to show that these can be derived from the theory of Galois functors.
\end{abstract}

\keywords{Entwining structures, bimonads, (generalised) Hopf modules, Galois functors,
 comparison functors}
\subjclass[2010]{16T05,  16T11, 18A40, 18C15}

\maketitle

\tableofcontents

\section*{Introduction}

Bialgebras $A$ over a commutative ring $R$
induce an endofunctor $A\ot_R-$  on the category $\M_R$ of $R$-modules
which has a monad and a comonad structure subject to some compatibility conditions.
To make the bialgebra $A$ a {\em Hopf algebra} the comparison functor from $\M_R$
to the category of Hopf modules $\M^A_A$ induced by $A\ot_R-$ has to be an equivalence
(e.g. \cite[7.9]{BBW}).

Since all these constructions are based on the tensor product in $\M_R$,
one may try to extend the notions to monads  $\bT=(T,m,e)$ on (strictly) monoidal categories $(\V,\ot, I)$.
To ensure that the Eilenberg-Moore category $\V_\bT$ is again monoidal,
$\bT$ has to be an {\em opmonoidal monad} (e.g. \cite{McC}).  Such functors consist of two parts:
the monad $\bT$  and a comonad $-\ot T(I)$ on $\V$ induced by the coalgebra $T(I)$  which are related
by a mixed distributive law (entwining) (e.g. \cite[Section 5]{MW-Gal}). Then a comparison functor between $\V$ and
the category of the {\em entwined  modules} determined by this entwining
(called {\em right Hopf $\bT$-modules} in \cite[Section 4.2]{BV}) may be considered.
\cite[Theorem 4.7]{MW-Op} gives
a necessary and sufficient condition for this comparison functor to be an equivalence
of categories.

In \cite{MW-Gal}, an entwining of a monad  $\bT=(T,m,e)$ and a comonad $\bG=(G,\delta, \ve)$ on any category $\A$ is considered, that is, a natural transformation  $\lambda:TG\to GT$ subject to certain commutativity conditions
(e.g. \cite[5.3]{W}). Then the comonad $\bG$
on $\A$ can be lifted to a comonad $\widehat{\bG}$
and the $\lambda$-entwining modules are just the $\widehat{\bG}$-comodules in $\A_\bT$ (see \ref{lambda-bim}).
For a comparison functor $K : \A \to (\A_T)^{\widehat{G}}$ one requires  commutativity of the
diagram
$$
\xymatrix{ \A  \ar[r]^-{K}\ar[rd]_{\phi_\bT} &
(\A_\bT)^{\widehat{G}} \ar[d]^{U^{\widehat{G}}}\\
& \A_\bT , }
$$
where $\phi_\bT$ denotes the free $\bT$-module functor and ${U^{\widehat{G}}}$ the forgetful $\widehat{G}$-comodule functor.
In \cite{MW-Gal,MW-Op} conditions are given which make $K$ an equivalence.

This setting comprises the opmonoidal monads outlined above and it also applies to the {\em bimonads}
on arbitrary categories introduced in \cite[5.13]{W}, \cite[Definition 4.1]{MW-Bim}.

To subsume the generalised Hopf modules studied by Aguiar and Chase  in
\cite{A}, one has to add an adjunction
$L \dashv R: \B\to \A$  to the picture and observe that the resulting adjunction  $\phi_{\bT}L \dashv RU_{\bT}$
 generates a comonad on  $ \A_{\bT}$.
 Now the results from \cite{MW-Gal} can be applied to the diagram
$$
\xymatrix{ \B \ar[r]^{K\quad} \ar[dr]_{\phi_\bT L} &
 (\A _{\bT})^{\widehat{\bG}} \ar[d]^{U^{\widehat{\bG}}}\\
   &\A_\bT  . }
$$
This is outlined in Section \ref{g-Hopf} leading to the Fundamental Theorem of generalised Hopf modules from \cite{A}.

Having made this extension, also the $A$-Hopf modules of a bimonoid $A$ in a
duoidal category $(\D, \circ,I,\ast,J)$ and the related comparison functor
considered by B\"ohm, Chen and Zhang in \cite{BCZ} can be handled in our setting:
 roughly speaking, for a bimonoid $A$, $-\circ A$ defines a monad  while
 $-\ast A$  is a comonad on $\D$ and the two functors are related by an entwining. Now it is fairly obvious how our techniques apply and
at the end of Section \ref{bimod} we obtain the Fundamental Theorem for $A$-Hopf modules from \cite{BCZ}.

\section{Galois functors} \label{sec-gen}

\begin{thm}\label{monad-comon}{\bf Monads and comonads.} \em
Let $\bT=(T, m, e)$ be a monad on a category $\A$. We write

\begin{itemize}
  \item
$\A_{\bT}$ for the Eilenberg-Moore category of $\bT$-modules and
$$\eta_{\bT}, \ve_{\bT}:\phi_{\bT} \dashv U_F :\A_{\bT} \to \A$$ for the corresponding forgetful-free adjunction;
\item
$\widetilde\A_\bT$ for the Kleisli category for   $\bT$ (as a full subcategory of $\A_{\bT}$, e.g. \cite{BBW})
and $\overline\phi_{\bT} \dashv u_{\bT}: \widetilde\A_\bT \to \A $ for the corresponding Kleisli adjunction.
\end{itemize}
\smallskip

Dually, if $\bG = (G, \delta, \ve)$ is a comonad on $\A$, we write
 $\A^{\bG}$ for the Eilenberg--Moore category of
$\bG$-coalgebras and
$$\eta^{\bG} , \ve^{\bG} :U^{\bG} \dashv \phi^{\bG} :\A \to \A^{\bG}$$ for the corresponding forgetful-cofree adjunction.
\end{thm}

\begin{thm}\label{com-fun}{\bf Comodule functors.} \em Consider an adjunction ${\eta, \sigma :F \dashv R:\A \to \B}$ and
a comonad $\bG=(G,\delta,\varepsilon)$ on $\A$. The functor $F: \B \to \A$ is called
a {\em left $\bG$-comodule} (e.g. \cite[Section 3]{MW-Bim}) if there exists a natural transformation $\kappa_F : F \to G F$
inducing commutativity of the diagrams

$$\xymatrix{
F \ar@{=}[dr] \ar[r]^-{\kappa_F}& GF \ar[d]^-{\varepsilon F}\\
& F,} \qquad \xymatrix{
F \ar[r]^-{\kappa_F} \ar[d]_-{\kappa_F}& GF \ar[d]^-{\delta F}\\
GF \ar[r]_-{G \kappa_F}& GGF.}
$$

There exist bijective correspondences between
\begin{itemize}
  \item [(i)] functors $K : \B \to \A^\bG$ with commutative diagram
$$\xymatrix{ \B  \ar[r]^-{K}\ar[rd]_{F} & \A^\bG \ar[d]^{U^\bG}\\
& \A \, ;}
$$
  \item [(ii)] left $\bG$-comodule structures $\kappa_{F} :F \to GF$ on
$F$;

   \item [(iii)] comonad morphisms $t_K:FR\to G$ from the comonad
generated by the adjunction $F \dashv R$ to $\bG$.
\end{itemize}
\medskip

\noindent These bijections are constructed as follows: Given a functor $K$ making the
diagram (\ref{com-fun}(i)) commute,
 $K(X)=(F(X), \kappa_{X})$ for some morphism $\kappa_{X}: F(X) \to GF(X)$
and the collection $\{\kappa_b,\, b \in\B\}$ constitutes a natural transformation
$\kappa_F:F \to GF$  making $F$ a $\bG$-comodule. Conversely, if $(F,\kappa_F: F\to GF)$
is a $\bG$-module, then $K : \B \to \A^\bG$ is defined by $K(X)=(F(X), (\kappa_F)_X)$.
Next, for any (left) $\bG$-comodule structure $\kappa_{F} :F \to GF$,
the composite
$$\xymatrix{t_{K}:FR \ar[r]^-{\kappa_F R}& GFR
\ar[r]^-{G \sigma }  & G}$$ is a comonad morphisms from the comonad
generated by  $F \dashv R$ to the comonad $\bG$.
On the other hand, for any comonad morphism $t:FR \to G$, the composite
$$\kappa_F: F \xrightarrow{ F\eta }FRF \xrightarrow{t F}GF$$
defines a $\bG$-comodule structure on $F$.

A left $\bG$-comodule functor $F$ is said to be {\em $\bG$-Galois} provided $t_K$
is an isomorphism (e.g. \cite[Definition 1.3]{MW-Gal}).
\end{thm}

\begin{proposition}\label{mes} {\rm (\cite[Theorem 4.4]{M})} The functor $K$
(in {\rm \ref{com-fun}}) is an equivalence of categories if and only if
\begin{itemize}
\item [(i)]  the functor $F$ is comonadic and
\item [(ii)] $t_K$ is an isomorphism ($F$ is  $\bG$-Galois).
\end{itemize}
\end{proposition}

\begin{thm} \label{mod-fun} {\bf Module functors.} \em
For a monad $\bT=(T,m,e)$ on $\A$, a ({\em left}) $\bT$-module functor consists
of a functor $R : \B \to \A$, equipped with a natural transformation
$\alpha_R : TR \to R$ satisfying
   ${\alpha_R} \cdot {eR} = 1$ and   ${\alpha_R} \cdot {mR} = {\alpha_R} \cdot {T\alpha_R }$.

If $(R, \alpha_R)$ is a $\bT$-module, then the assignment
$$X \longmapsto (R(X), (\alpha_R)_X)$$
extends uniquely to a functor $K':\B \to \A_\bT$
with $U_\bT K'=R$. This gives a bijection, natural in $\bT$, between
left $\bT$-module structures on $R:\B \to \A$ and the functors $K':\B \to \A_\bT$
with $U_\bT K=R$.

For any $\bT$-module $(R: \B \to\A,\alpha_R)$ admitting a left adjoint functor
$F : \A \to \B$, the composite
$$ t_{K'}: \xymatrix{T \ar[r]^-{T \eta}& TRF \ar[r]^-{\alpha_R F} & RF},$$
where $\eta : 1 \to RF$ is the unit of the adjunction $F \dashv R$,
is a monad morphism from $\bT$ to the monad on $\A$
generated by the adjunction $F \dashv R$.

A left $\bT$-module $R: \B \to \A$ with a left
adjoint $F: \A \to \B$ is said to be $\bT$-\emph{Galois} if the
corresponding morphism $t_{K'}: T \to RF$ of monads on
$\A$ is an isomorphism.
\end{thm}

Expressing the dual of \cite[Theorem 4.4]{M} in the present situation gives:

\begin{proposition}\label{mes1}
The functor $K'$ (in {\rm\ref{mod-fun}}) is an equivalence of categories if and
only if
\begin{itemize}
\item [(i)] the functor $R$ is monadic and
\item [(ii)] $R$ is a $\bT$-Galois module functor.
\end{itemize}
\end{proposition}

\begin{thm}\label{mixed}{\bf Mixed distributive laws.} \em
Let  $\bT=(T,m,e)$ be a monad and $\bG=(G, \delta,\varepsilon)$ a comonad on a
category $\A$.

A \emph{mixed distributive law} or
\emph{entwining} from $\bT$ to $\bG$ is a natural
transformation $\lambda : TG \to TG$ with certain commutative diagrams
(e.g. \cite{Wf}, \cite[5.3]{W}).

A \emph{lifting} of  $\bG$ to $\A_\bT$ is a comonad
$\widehat{\bG}=(\wG, \widehat{\delta},\widehat{\varepsilon})$
on $\A_\bT$ for which $GU_\bT=U_\bT \wG$, $ U_\bT \widehat{\delta}=\delta U_\bT$
and $ U_\bT \widehat{\ve}=\ve U_\bT$.
\end{thm}

The following is a version of \cite[Theorem 2.2]{Wf}:

\begin{theorem}\label{wolff} Let $\bT=(T,m,e)$ be a monad and $\bG=(G, \delta,\varepsilon)$ a comonad on a category $\A$.
Then there is a one-to-one correspondence between
\begin{itemize}
\item
 mixed distributive laws $\lambda : TG \to TG$  from $\bT$ to $\bG$ and
\item
 liftings   of   $\bG$ to a comonad $\widehat{\bG}$ on $\A_\bT$.
\end{itemize}
\end{theorem}

To obtain a lifting $\widehat{\bG}$ from a distributive law $\lambda$, one defines  for  $(X,h)\in \A_\bT$,
 $\widehat{G}(X,h)$ as the $\bT$-module
$$(G(X),\, TG(X)\xr{\lambda_X} GT(X) \xr{G(h)} G(X)).$$

Conversely, if one has a lifting comonad $\widehat{\bG}$, one defines
$\lambda : TG \to TG$ by
$$\xymatrix{ TG \ar@{-->}[rd]_\lambda\ar[r]^-{TGe} & TGT \ar@{=}[r] &TGU_\bT\phi_\bT TU_\bT\widehat{G}\phi_\bT \ar@{=}[r] &
U_\bT\phi_\bT U_\bT\widehat{G}\phi_\bT \ar[d]_-{U_\bT \ve_\bT\widehat{G}\phi_\bT}\\
 &GT \ar@{=}[r]&G U_\bT \phi_\bT \ar@{=}[r] &  U_\bT\widehat{G}\phi_\bT.}$$

\begin{thm}\label{lambda-bim}{\bf $\lambda$-bimodules.} \em
We write $\A^{\bG}_{\bT}(\lambda)$ (or just $\A^{\bG}_{\bT}$ when $\lambda$ is understood) for the category whose
objects are triples $(X, h, \theta)$, where $(X, h) \in
\A_{\bT}$ and $(X,\theta) \in \A^{\bG}$, with commuting diagram (e.g. \cite[5.7]{W})
$$
\xymatrix{
T(X) \ar[r]^-{h} \ar[d]_-{T(\theta)}& X \ar[r]^-{\theta}& G(X) \\
TG(X) \ar[rr]_-{\lambda_X}&& GT(X). \ar[u]_-{G(h)}}
$$

The assignment $(X, h, \theta) \to ((X, h), \theta)$  yields an isomorphism of categories
$$\A^{\bG}_{\bT}(\lambda) \simeq (\A _{\bT})^{\widehat{\bG}}.$$
\end{thm}

\begin{thm} \label{com-functor}{\bf Generalised Galois functor.} \em
With the data as given in Theorem \ref{wolff},
let $\lambda: TG \to GT$ be a mixed distributive law.
Given an adjunction $\nu, \varsigma : L \dashv R :\B \to \A$,
assume $K: \B \to (\A _{\bT})^{\widehat{\bG}}$ to be a functor with
$U^{\widehat{\bG}}K=\phi_\bT L$, i.e.\ with commutative diagram
\begin{equation}\label{tri}
\xymatrix{ \B \ar[rr]^{K} \ar[d]_{L} &&
(\A _{\bT})^{\widehat{\bG}} \ar[d]^{U^{\widehat{\bG}}}\\
\A   \ar[rr]^{\phi_\bT}&&\A_\bT .  }
\end{equation}
Write $\mathcal{G}'$ for the comonad on the category $\A_{\bT}$ generated by the adjunction
$$\phi_{\bT}L \dashv RU_{\bT}: \A_{\bT}\to \B$$ and write
$t_K : \mathcal{\bG}' \to \widehat{\bG}$
for the corresponding comonad morphism (see \ref{com-fun}).
\end{thm}

Applying Proposition \ref{mes} to the present situation gives:

\begin{theorem}\label{comparison} In the setting of \, {\rm\ref{com-functor}},
the functor $K: \B \to (\A_\bT)^{\widehat{\bG}}$ is an equivalence of categories if
and only if
\begin{itemize}
  \item [(i)]    the functor $\phi_\bT L : \B \to \A_\bT$ is comonadic and
  \item  [(ii)]  $\phi_\bT L$ is a $\widehat{\bG}$-Galois comodule functor.
\end{itemize}
\end{theorem}
The following proposition gives a necessary and sufficient condition for the
functor $\phi_\bT L$ to be $\widehat{\bG}$-Galois
(generalising \cite[Proposition 2.10]{MW-Op}).

\begin{proposition}\label{galois} In the setting of \, {\rm\ref{com-functor}},
$\phi_\bT L$ is $\widehat{\bG}$-Galois if and only if the natural transformation $t_K\phi_\bT$ is an isomorphism.
\end{proposition}
\begin{proof} One direction is clear, so suppose that $t_K\phi_\bT$ is an isomorphism.

Let $\kappa: \phi_\bT L \to \wG \phi_\bT L$ be the left $\wG$-comodule structure on $\phi_\bT L$
corresponding to the diagram (\ref{tri}). Then, by \ref{com-fun}, $t_K : \mathcal{G}' \to \widehat{\bG}$ is the composite
$$\phi_\bT L RU_\bT \xr{\kappa RU_\bT}\wG \phi_\bT LRU_\bT \xr{\wG\phi_\bT \varsigma U_\bT} \wG\phi_\bT U_\bT \xr{\wG\ve_{\!\bT}} \wG.$$
Consider the natural transformation $U_\bT t_K$
$$U_\bT\phi_\bT L RU_\bT\xr{U_\bT\kappa RU_\bT}U_\bT \wG \phi_\bT LRU_\bT \xr{U_\bT \wG \phi_\bT \varsigma U_\bT} U_\bT \wG \phi_\bT U_\bT \xr{U_\bT \wG \ve_{\!\bT}} U_\bT\wG,$$
and, using  $U_\bT \wG= G U_\bT$,   rewrite it as
$$ U_\bT \phi_\bT L RU_\bT \xr{U_\bT \kappa RU_\bT} G U_\bT \phi_\bT LRU_\bT \xr{G U_\bT \phi_\bT \varsigma U_\bT} G U_\bT \phi_\bT U_\bT  \xr{G U_\bT \ve_{\!\bT}} G U_\bT. $$
By \cite[Lemma 2.19]{BLV},  if $U_\bT t_K\phi_\bT$ is an isomorphism, then $U_\bT t_K$ is so. But since $U_\bT$ is conservative,
$t_K$ is an isomorphism, too. This completes the proof.
\end{proof}

In view of Theorem \ref{comparison}, it is desirable to find sufficient conditions for
the composite $\phi_\bT L$ to be comonadic. The next proposition gives two such conditions.

\begin{proposition}\label{comonadicity} In the setting of {\rm\ref{com-functor}}, suppose that $\A$ is Cauchy complete and  $L$ is comonadic. Then
the composite $\phi_\bT L$ is comonadic under any of the conditions
\begin{rlist}
  \item   the unit $e:I\to T$ is a split monomorphism, i.e. there is a natural
   transformation $e': T \to I$ with $e'e=1$;
  \item   the monad $\bT$ is of effective descent type
    ($\phi_\bT  : \A \to \A_\bT$ is comonadic)
  and $\A$ has and $LR$ and $LR LR$ preserve equalisers of coreflexive $T$-split pairs.
\end{rlist}
\end{proposition}
\begin{proof}(i)  Since
\begin{itemize}
  \item $\A$ is Cauchy complete,
  \item $e:1 \to T$ is a split monomorphism, and
  \item $e$ can be seen as the unit of the adjunction $\phi_\bT \dashv U_\bT$,
\end{itemize} it follows from \cite[Proposition 3.14]{Me} that any $\phi_\bT$-split
pair is part of a split equaliser in $\A$, and thus the functor $\phi_\bT$ creates
equalisers of
$\phi_\bT$-split pairs as split equalisers in $\A$, from which it follows that
$\phi_\bT L$ is comonadic whenever $L$ is so.

(ii)  Since the functor $L$ is assumed to be comonadic, to say that
 $\A$ has and $LR$ and $LRLR$
preserve equalisers of coreflexive $T$-split pairs is to say
that the functor $L$ creates equalisers of those pairs whose image under $L$ is
part of a $T$-split
equaliser (see, for example, \cite[Proposition 4.3.2]{Bo}).
Since $\A$ has and $T$ preserves equalisers of
coreflexive $T$-split pairs if and only if $\A$ has and $\phi_\bT$
preserves equalisers of coreflexive $\phi_\bT$-split
pairs (\cite[Proposition 3.11]{Me}), and since $\bT$ is of
effective descent type by hypothesis, it
follows that $\B$ has and the composite $\phi_\bT L$ preserves
equalisers of coreflexive $\phi_\bT L$-split pairs.
Using now the fact that $\phi_\bT L$, being a composite of two conservative functors,
 is conservative,
the result follows from the dual of Beck's monadicity theorem (see \cite{Mc}).
\end{proof}

For later use (in \ref{split1}) we prove the following technical observation.

\begin{proposition}\label{split.comon} Let $\A$ be Cauchy complete and
  ${L \dashv R:\A \to \B}$  an adjunction whose unit is a split monomorphism.
Then, in any commutative (up to isomorphism) diagram
$$
\xymatrix{
\A \ar[r]^-{F} \ar[dr]_-{R} & \mathbb{E} \ar[d]^-{K}\\
& \B , }
$$
\begin{rlist}
\item
the functor $F:\A \to \mathbb{E}$ is conservative;
\item any coreflexive $F$-split pair of morphisms has a split equaliser in $\A$;
\item the functor $F$ is comonadic if and only if it has a right adjoint.
\end{rlist}
\end{proposition}
 \begin{proof} (i) Since $\A$ is Cauchy complete and since the unit of the adjunction
is a split monomorphism, the
functor $R$ is comonadic (e.g., \cite[Proposition 3.16]{Me}) and, in particular,
conservative. This implies -- since
$KF$ is isomorphic to $R$ -- that  $F$ is conservative, too.

(ii) Suppose that $\xymatrix{ X \ar@<2pt>[r]^f \ar@<-2pt>[r]_g& Y }$ is
an $F$-split pair of morphisms in $\A$.
Then the morphisms $F (f)$ and $F(g)$ have a split equaliser in $\mathbb{E}$,
so that the pair $(F (f), F (g))$ is
\emph{contractible} (see \cite{Mc}). Since contractible pairs, being
equationally defined, are preserved by any functors,
 the pair $(R(f), R(g))$, being isomorphic to the pair $(KF(f), KF(g))$, is also
contractible. But since
$\A$ is Cauchy complete and since the unit of the adjunction is a split
monomorphism (which just means that
the functor $R$ is $1_\A$-\emph{separable}), it follows from \cite[Proposition 3.8]{Me}
 that the pair $(f,g)$
is contractible, too.
Then,  $\A$ being Cauchy complete, $f$ and $g$ have a
split equaliser (e.g. \cite{B}) and this equaliser is clearly preserved by $F$.
\smallskip

(iii) follows from the fact that  split equalisers are preserved by any functor.
\end{proof}

\section{Generalised Hopf modules}\label{g-Hopf}

In \cite{A}, Aguiar and Chase studied generalised Hopf modules in monoidal
categories and proved a {\em Fundamental Theorem} for them.
In this section we show how this result can be obtained as a special case of
our approach.
We first recall  elementary facts about modules and comodules in a monoidal category
(e.g. \cite{M}, \cite{Pa}).

\begin{thm}\label{mon-comon}{\bf Monoids and comonoids in monoidal categories.} \em
Let $(\V, \otimes, I, a, l, r)$ be a monoidal category, where  $a, l, r$ are the
associativity, left identity,
and right identity isomorphisms for the monoidal structure on $\V$.

A \emph{monoid} in  $\V$ consists of an object $A$ of $\V$ endowed with a
multiplication $m:A \otimes A \to A$  and a unit morphism $e : I \to A$ such
that the usual identity
and associative conditions are satisfied. A monoid morphism $f:A \to A'$ is a morphism
in $\V$  preserving $m$ and $e$. The category of monoids in $\V$ is denoted by
$\text{\textbf{Mon}}(\V)$.

Given a monoid $(A, e_A, m_A)$ in $\V$, a \emph{left} $A$-\emph{module} is a
pair $(V, \rho_V)$, where $V$ is
an object of $\V$ and $\rho_V: A \otimes V \to V$ is a morphism in $\V$,
called the \emph{$A$-action} on $V$, such that
$\rho_V(m_A \otimes V)a_{A,A,V}=\rho_V(A\otimes\rho_V)$ and $\rho_V (e_A\otimes V)=l_V$.

For any monoid $A$ in $\V$, the left $A$-modules are the objects of a
category ${_{A}\!\V}$.
A morphism $f:(V,\rho_V)\to (W,\rho_W)$ is a morphism $f: V \to W$ in $\V$ such that
$\rho_W (A \otimes f)=f\rho_V$.
Analogously, one has the category $\V\!_{A}$ of right $A$-modules.

Let $A$ and $B$ be two monoids in $\V$. An object $V$ in $\V$ is called an
$(A,B)$-\emph{bimodule} if there are morphisms $\rho_V: A \otimes V \to V$ and
$\varrho_V: V \otimes B \to V$ in $\V$ such that
$(V, \rho_V) \in {_{A}\!\V}$, $(V, \varrho_V) \in \V\!_{B}$ and
$\varrho_V(\rho_V\otimes B)=\rho_V (A \otimes \varrho_V)a_{A,V,B}$.
A morphism of $(A,B)$-bimodules is a morphism in $\V$ which is a morphism
of both the left $A$-modules and right
$B$-modules. Write $_{A}\!\V\!_{B}$ for the corresponding category.

\emph{Comonoids} and (\emph{left, right, bi-}) \emph{comodules} in $\V$ can
be defined as monoids
and left (right, bi-) modules in the \emph{opposite} monoidal category
$(\V^{\text{op}}, \otimes, I, a^{-1}, l^{-1}, r^{-1})$.
The resulting categories are denoted by $\text{\textbf{Comon}}(\V)$, $^C\V$, $\V^C$
and $^C\V^{C'}$, $C$ and $C'$ being comonoids in $\V$.
\end{thm}

\begin{thm}\label{tensor}{\bf Tensor product of modules.} \em
If $A$ is a monoid  in $\V$, $(V,\varrho_V) \in \V_{A}$ a right
$A$-module, and $(W,\rho_W) \in {_{A}\!\V}$ a left $A$-module, then their
\emph{tensor product} (\emph{over} $A$) is defined as the object part of the
 coequaliser (if this exists)
\begin{equation}\label{tensorD}
{\xymatrix{(V \otimes A) \otimes W \ar@/_1.6pc/[rr]_-{\varrho_V \otimes W}
\ar[r]^-{a_{V,A,W}} &
V \otimes  (A \otimes W)\ar[r]^-{V \otimes \rho_W}& V  \otimes  W
\ar[r]^-{\text{can}}& V  \otimes_A W.}}
\end{equation}

Given  another left $A$-module $(W', \rho_{W'})$ for which  $V  \otimes_A W'$ exists,
and a morphism $f:W \to W'$ of left $A$-modules, we form the diagram
$${\xymatrix{ (V \otimes  A) \otimes  W \ar[d]|-{(V \otimes  A) \otimes f}
\ar@/^1.6pc/[rr]^-{\varrho_V \otimes W}
 \ar[r]_-{a_{V,A,W}} & V \otimes  (A \otimes  W) \ar[d]|{V \otimes (A \otimes f)}
\ar[r]_-{V \otimes \rho_W}  & V
 \otimes  W \ar[d]|-{V  \otimes f}\ar[r]^-{\text{can}}&
V \otimes_A W \ar@{..>}[d]|-{V  \otimes_A  f}\\
(V \otimes  A) \otimes  W' \ar@/_1.6pc/[rr]_-{\varrho_V \otimes W'}
 \ar[r]^-{a_{V,A,W'}} &
V \otimes (A \otimes  W')\ar[r]^-{V \otimes \rho_{W'}}&
V\otimes  W' \ar[r]_-{\text{can}}& V  \otimes_A W',}}$$
 in which
\begin{itemize}
  \item  $(V\otimes f)(\varrho_V\otimes W)=(\varrho_V \otimes W')((V\otimes A)\otimes f)$
by functoriality of $\otimes$,
  \item  the left square commutes by naturality of $a$, and
  \item the middle square commutes because $f$ is a morphism of left $A$-modules;
\end{itemize}
from this one sees that there ia a unique morphism
${V\otimes_A f:V\otimes_A W \to V\otimes_A W'}$ making the right square commute.
It is easy to see that if for ${W''\in {_{A}\V}}$, the tensor product
${V  \otimes_A W'}$ exits, then for any
morphism $g:W' \to W''$ in ${_{A}\V}$,
$${V  \otimes_A  (gf)=(V  \otimes_A  g)(V  \otimes_A  f)}.$$

If $B$ is another monoid in $\V$ such that the functors
$B \otimes -, B \otimes (B \otimes -):\V \to \V$ both preserve the equaliser
(\ref{tensorD})
and if $V \in {_B\V_A}$, then the tensor product $V  \otimes_A  W$
has the structure of a left $B$-module such that
$\text{can}:V \otimes W \to V \otimes_A W$ becomes a morphism of left $B$-modules.
Moreover, if these functors also preserve the equaliser defining $V \otimes_A W'$,
then $V \otimes_A  f$ also becomes a left $B$-module morphism.

Recall (for example, from \cite{Pa}) that the forgetful functor
$${_{A}U} \colon {_{A}\!\V} \to \V,\,\; (V,\rho_V) \mapsto V,$$
is right adjoint, with the left adjoint ${_{A}\phi}:\V \to {_{A}\!\V}$ sending
each $V \in \V$ to the "free" left $A$-module
$$(A\otimes V,\, A\otimes (A \otimes V)\xr{a_{A,A,V}}(A\otimes A)\otimes V
\xr{m \otimes V} V\otimes V).$$

Write ${_A\!\bT}$ for the monad on $\V$ generated by the adjunction
${{_A \phi} \dashv {_A \!U} : {_A\!\V} \to \V}$.
It is well known that the corresponding Eilenberg-Moore category $\V\!_{_{_A}\!\!\bT}$
of ${_A\!\bT}$-modules is exactly the category $_{A}\!\V$ of left $A$-modules.
\end{thm}

\begin{lemma}\label{tensor.pr} Let $A$ be a monoid  in $\V$ and  ${M =A \otimes V}$
 the free left $A$-module generated by $V \in \V$. Then
\begin{zlist}
\item  for any ${N\in \V\!_A}$, the tensor product ${N \! \otimes_A \!M }$ exists
and is isomorphic to $N \!\otimes V$;
\item for $N \in {_B\!\V_A}$, $B$  any monoid in $\V$,
 ${N \!\otimes_A \!M }$ is a left $B$-module;
\item  for any morphism $f:A \otimes V \to A \otimes V'$ in $\V$,
  $N  \otimes_A f$ is a morphism of left $B$-modules;
\item   for any morphism $: V \to V'$ in $\V$,
the induced morphism $N  \otimes_A (A \otimes g)$ of left $B$-modules
is isomorphic to $N  \otimes g$.
\end{zlist}
\end{lemma}
\begin{proof} Everything follows from \ref{tensor} and the fact that the equaliser
defining the tensor product
$N\otimes_A M$ is split and thus is preserved by any functor.
\end{proof}

\begin{remark}\label{remark}\em
The full subcategory of $_{A}\!\V$ generated
by the left $A$-modules of the form ${A \otimes  V}, \, V \in \V,$ is just the
Kleisli category
$\widetilde{\V}\!_{_A\!\bT}$ of the monad ${_A\!\bT}$ (e.g. \cite[2.4]{BBW}).
Hence Lemma \ref{tensor.pr} may be alternatively stated as follows:

{\em
Let $N \in {_B\V_A}$. Then, for any $X \in \widetilde{\V}\!_{_A\!\bT}$,
 the tensor product $N\! \otimes_A \!X$ exists and has
the structure of a left $B$-module. So the assignment $X \mapsto N\!\otimes_A \!X$
yields a functor
$N \!\otimes_A -: \widetilde{\V}\!_{_A\!\bT} \to {_B\!\V}$ leading to the
commutative diagram
$$
\xymatrix{
\V \ar[dr]_{N \otimes -} \ar[r]^-{\overline{\phi}_{_A\!\bT}} &
\widetilde{\V}\!_{_A\!\bT} \ar[d]^{N \otimes_A -}\\
&  {_B\V}\, .}$$   }
\end{remark}

\begin{thm}\label{bimon}{\bf Opmonoidal functors.} \em Recall that - following  \cite{McC} -
 an \emph{opmonoidal functor} from a monoidal
category $(\V, \otimes, I)$ to a monoidal category $(\V', \otimes', I')$ is a triple
$(S, \omega, \xi)$, where $S: \V \to \V'$ is a functor,
$\omega_{V,W} : S(V\otimes W)\to S(V)\otimes S(W)$
is a natural transformation between functors $\V \times \V \to \V$,
and $\xi: S(I) \to I'$ is a morphism compatible with the tensor structures.
Note that opmonoidal functors $S$ take $\V$-comonoids (i.e. comonoids in $\V$) into
$\V'$-comonoids in the sense that if $C=(C, \delta, \varepsilon)$ is a $\V$-comonoid,
then it produces a $\V'$-comonoid
$$S(C)=(S(C),\,\omega_{C,C} \cdot S(\delta),\, \xi \cdot S(\ve)).$$

In \cite{McC}, an  \emph{opmonoidal monad}
on a monoidal category $(\V,\otimes, I)$ is defined as a monad $\bT=(T,m, e)$ on $\V$
such that the functor $T$ and the natural transformations $m$ and $e$ are opmonoidal.
Such monads are also called {\em bimonads} (e.g. in \cite[Definition 3.2.1]{A}) but
they are different from what are called {\em bimonads} in \cite[Definition 4.1]{MW-Bim}
(compare \cite[Section 5]{MW-Gal}).
\smallskip

The basic property of opmonoidal monads $\bT$ is that they
lead to a monoidal structure on the Eilenberg-Moore category $\V_\bT$ of
$\bT$-modules in such a way that the forgetful functor $U_\bT: \V_\bT \to \V$ is
strictly monoidal.
Explicitly, for $\bT$-modules $(V,h)$ and $(W,g)$, the tensor product
$(V,h) \otimes (W,g)$
is given by
$$(V \otimes W, \,\,T(V \otimes W) \xr{\omega_{V,W}} T(V) \otimes T(W)
\xr{h \otimes g}V \otimes W)$$
and the unit object of $\V_\bT$ is the $\bT$-module $(I,\, \xi: T(I)\to I)$.
The unitary and associativity
isomorphisms for $\V_\bT$ are inherited from $\V$.
\end{thm}

\begin{thm}\label{alg-comonoid}{\bf $\bT$-module-comonoids.} \em
Given an opmonoidal monad $\bT$ on $\V$, a comonoid in the monoidal category $\V_\bT$
is called a $\bT$-\emph{module-comonoid}.
Explicitly, a $\bT$-module-comonoid $\bZ=((Z,\sigma), \delta, \ve)$ consists
of an object $(Z,\sigma)\in \V_\bT$ and $\V$-morphisms $\delta: Z \to Z \otimes Z$
and $\ve: Z \to I$ such that
$U_\bT(\bZ)=(Z, \delta, \ve)$ is a $\V$-comonoid and that $\delta$ and $\ve$ are
morphism of $\bT$-modules.

For any $\V$-comonoid $(C,\delta,\varepsilon)$,  $T(C)$ allows for a module-comonoid
structure with the morphisms (e.g. \cite{A})
\begin{itemize}
  \item $TT(C) \xr{m_C}T(C)$,
  \item $T(C) \xr{T(\delta)} T(C \otimes C) \xr{\omega_{C,C}} T(C) \otimes T(C)$,
  \item $T(C) \xr{T(\varepsilon)}T(I) \xr{\;\xi\;}I .$
\end{itemize}
 We write $\bT(C)$ for this module-comonoid.
\end{thm}

\begin{thm}\label{action}{\bf $\V$-categories.} \em A \emph{left} $\V$-\emph{category}
is a category $\A$ equipped with a bifunctor
$$-\diamond-: \V \times \A \to \A, $$
called the \emph{action of} $\V$ \emph{on} $\A$,
and invertible natural transformations
$$\alpha_{V,W,X}:(V \otimes W)\diamond X \to V \diamond (W \diamond
X)\,\,\; \text{and}\,\,\; \lambda_X : I \diamond X \to  X,$$
called the \emph{associativity} and \emph{unit} isomorphisms,
respectively, satisfying two coherence axioms (see B\'{e}nabou
\cite{Bn}). Note that $\V$ has a canonical (left) action on itself, given by taking
$V \diamond W=V\otimes W$, $\alpha=a $, and $\lambda=l$.
\smallskip

Given a left $\V$-category $\A$ and a monoid $(A, e_{A}, m_{A})$
in $\V$, one has a monad $\bbT$ on $\A$ defined on any $X \in\A$ by
\begin{rlist}
\item[$\bullet$] $\bbT(X)=A \diamond X$,
\item[$\bullet$] $(e_{\bbT})_X : X \xr{\lambda_{X}^{-1}} I
\diamond X \xr{e_{A} \diamond X} A \diamond X =\bbT(X)$,
\item[$\bullet$] $(m_{\bbT})_X :\bbT(\bbT(X))=A\diamond (A \diamond X)
     \xr{\alpha^{-1}_{A,A,X}}
     (A \otimes A)\diamond X \xr{m_A\diamond X}  A \diamond X = \bbT(X),$
\end{rlist}
and we write $_A\A$ for the Eilenberg-Moore category $\A_{\bbT}$ of
$\bbT$-modules. For the canonical
left action of $\V$ on itself, $_A\A$ is just  the
category $_{A}\!\V$ of (left) $A$-modules.\smallskip

Dually, for any $\V$-coalgebra $(C,\varepsilon_C,\delta_C)$, the endofunctor
$C \diamond -:\A \to \A$
is the functor-part of a comonad $\bG^l\!\!_C$ on $\A$
and one has the corresponding Eilenberg-Moore category $^C\!\A=\A^{{\bG^l\!\!_C}}$;
for $\A=\V$ this is just the category $^{C}\V$ of (left) $C$-comodules.
We sometimes write $_A\phi$ and
$^C\!\phi$ for the functors $\phi_{\bT^l_A}$ and $\phi_{\bG^l_C}$, respectively.
\smallskip

Symmetrically, one has the monad $\bT^r_A=-\diamond A$
(resp. comonad $\bG^r_C=-\diamond C$) on $\A$,
the corresponding Eilenberg-Moore category $\A_A$ (resp. $\A^C$) of $\bT^r_A$-modules
(resp. $\bG^r_C$-comodules),
and the functor $\phi_A :\A \to \A_A$ (resp. $\phi^C :\A \to \A^C$).
\end{thm}

\begin{thm}\label{comodule}{\bf Comodules over opmonoidal functors.} \em
Let $-\diamond-: \V \times \A \to \A$
be a left action of a monoidal category $\V$ on a category $\A$ and
let $\mathcal{F}:\V \to \V$ be an opmonoidal functor on $\V$.
A \emph{comodule} over $\mathcal{F}$ is a pair $(H, \chi)$, where
$H: \A \to \A$ is a functor and
$\chi_{V,X}:H(V \diamond X) \to \mathcal{F}(V) \diamond H(X)$
is a natural transformation satisfying two axioms (e.g. \cite[Definition 3.3.1]{A}).

Suppose that $\bT=(T, m_{\bT},e_{\bT})$ is an opmonoidal monad on $\V$ (with structure
$\omega_{V,W}:T(V\otimes W)\to T(V)\otimes T(W)$
and $\xi: T(I)\to I$) and that $\bS=(S, m_{\bS},e_{\bS})$ is a monad
on $\A$ such that the functor $S$ is a comodule over the opmonoidal functor
$(T, \omega,\xi)$ via
$\chi_{V,a}:S(V \diamond a) \to T(V) \diamond S(a)$.
One says that $(\bS, \chi)$ is a \emph{comodule-monad} over the bimonad
$\bT$ if $\chi$ is compatible with the monad structures (\cite[Definition 3.5.1]{A}).
Considering $\bT$ as a monad on the left $\V$-category $\V$,
it follows from the definition of an opmonoidal monad that the pair $(\bT, \omega)$
is a comodule-monad over the opmonoidal monad $\bT$.

There is a left action of the monoidal category $\V_\bT$ (with the monoidal structure
from \ref{bimon})
on the category $\A_\bS$: given a $\bT$-module $(V,f)$ and an
$\bS$-module $(X,h)$, $(V,f)\diamond (X,h)$ is the pair
(e.g. \cite[Proposition 3.5.3]{A})
$$(V \diamond X, \, S(V \diamond X) \xr{\chi_{V,X}}T(V) \diamond S(X)
 \xr{f \,\diamond\, h} V \diamond X).$$
\end{thm}

\begin{thm} \label{ass.1}{\bf Assumption 1.} \em
We henceforth suppose that
\begin{itemize}
  \item $\bT=((T, m_{\bT},e_{\bT}),\omega,\xi)$ is an opmonoidal monad on a
  monoidal category $(\V, \otimes, I, a, l, r)$;
  \item $(\A, \diamond, \alpha,\lambda)$ is a left $\V$-category;
  \item $\bS=(S, m_{\bS},e_{\bS})$ is a $\bT$-comodule-monad on $\A$ via
    $$\chi_{V,-}:S(V \diamond -) \to T(V) \diamond S(-);$$
  \item $\bZ=((Z,\sigma),\delta,\varepsilon)$ is a $\mathcal{T}$-module-comonoid.
\end{itemize}
\end{thm}
\smallskip

Since $\bZ$ is a comonoid in the monoidal category $\V_\bT$ and since $\V_\bT$
acts from the left on $\A_\bS$, one has the $\A_\bS$-comonad $\bbG$.
Moreover, since $\bZ_0=U_\bT(\bZ)$ is a comonoid in
the monoidal category $\V$, one has the $\A$-comonad $\bG_{\bZ_0}^l$,
and it is not hard to check that $\bbG$ is a lifting of $\bG_{\bZ_0}^l$ to $\A_\bS$.
It follows from Theorem \ref{wolff} that there is a mixed distributive law $\lambda$
from the $\A$-monad $\bS$ to the $\A$-comonad $\bG_{\bZ_0}^l$.

\begin{proposition}\label{entwining} With the data considered in {\rm \ref{comodule}},
$\lambda$ is the composite
$$S(Z\diamond -)\xr{\chi_{Z,-}}T(Z)\diamond S(-)\xr{\sigma\diamond S(-)}Z\diamond S(-).$$
\end{proposition}
\begin{proof} By \ref{comodule},  for any $(X,h)\in \A_\bT$,
$$\bbG(X,h)=(Z \diamond X,\, S(Z \diamond X)\xr{\chi_{Z,X}}T(Z)\diamond S(X)
\xr{\sigma \diamond h} Z \diamond X),$$
and it follows that $(\ve_\bS)_{\bbG(X,h)}=(\sigma \diamond h)\cdot \chi_{Z,X}$,  thus
$$(\ve_\bS)_{\bbG \phi_\bS(X)}=(\sigma \diamond (m_\bS)_X) \cdot \chi_{Z,S(X)},$$
since $\phi_\bS(X)=(S(X), (m_\bS)_X)$.
By Theorem \ref{wolff}, for $X \in \A$, $\lambda_X$ is the composite
$$S(Z \diamond X) \xr{S(Z \diamond (e_\bS)_X)} S(Z \diamond S(X))
\xr{\chi_{Z,S(X)}}T(Z) \diamond SS(X)
\xr{\sigma \diamond (m_\bS)_X} Z \diamond S(X).$$
In the  diagram
$$
\xymatrix{
S(Z \diamond X) \ar[d]_-{\chi_{Z,X}} \ar[rr]^{S(Z \diamond (e_\bS)_X)} &&
S(Z \diamond S(X)) \ar[d]^-{\chi_{Z,S(X)}}\\
T(Z) \diamond S(X) \ar[rrd]_{\sigma \diamond S(X)} \ar[rr]^{T(Z) \diamond (e_\bS)_X} &&
T(Z) \diamond SS(X) \ar[d]^{\sigma \diamond (m_\bS)_X}\\
&& Z \diamond S(X)\, ,}$$
the rectangle commutes by naturality of $\chi$, while  $m_\bS \cdot e_\bS=1$
implies commutativity of the triangle;
it follows that $\lambda_X=(\sigma \diamond S(X))\cdot \chi_{Z,X}$.
\end{proof}

 \begin{thm}\label{gen-Hopf}{\bf Generalised Hopf modules.} \em
$^{\bZ} \!(\A_{\bS})=\A_\bS^{\!\bG^l_{\bZ_0}}(\lambda)$ is the category of
 $\lambda$-bimodules (see \ref{lambda-bim}); the objects
 are  triples $(X, h,\vartheta)$, where $X \in \A$,
$(X,\, h:S(X) \to X)\in \A_\bS$,
$(X,\, \vartheta : X \to Z \diamond X) \in {^{\bZ_0}\!\A}$
with commuting diagram
$$
\xymatrix{ S(X) \ar[d]_{S(\vartheta)} \ar[r]^{h} &X \ar[r]^{\vartheta\quad} &
 Z \diamond X  \\
S(Z \diamond X) \ar[rr]_{\lambda_X}&& Z \diamond S(X) \ar[u]_-{Z\diamond h}.}
$$

In \cite[Definition 3.6.1]{A},
these are called {\em generalised Hopf modules} and the category $^{\bZ} \!(\A_{\bS})$
is denoted by $\text{{\rm Hopf}}(\mathcal{T}, \mathcal{S}, Z)$.
\end{thm}

\begin{thm} \label{ass.2}{\bf  Assumption 2.} \em
We henceforth suppose that $(C,\delta,\varepsilon)$ is a comonoid in $\V$ and
that $\bZ=\mathcal{T}(C)$ is the corresponding $\mathcal{T}$-module-comonoid
(see \ref{alg-comonoid}).
\end{thm}

\begin{lemma} \label{lemma1} In the situation considered above, the assignment
$$(X, \theta) \rightsquigarrow (S(X), (m_\bS)_X, \vartheta),$$
where $\vartheta: S(X) \to Z\diamond S(X)$ is the composite
$$S(X) \xrightarrow{S(\theta)} S(C\diamond X)
\xrightarrow{\chi_{C,X}}\mathcal{T}(C)\diamond S(X)=Z\diamond S(X),$$
yields a functor
$$K: {^C\!\A} \to  {^\bZ\!(\A_{\bS}) } $$ yielding  commutativity  of the diagram
\begin{equation}\label{lemma.d}
\xymatrix{ ^C\!\A \ar[rr]^-{K} \ar[d]_{^C\!U} && {^\bZ\!(\A_{\bS})}  \ar[d]^{^\bZ \!U} \\
\A \ar[rr]_{\phi_{\bS}}&& \A_{\bS}.}
\end{equation}
\end{lemma}
\begin{proof}
To show that $(X, \vartheta)\in {^{\bZ_0}\!\A}$ is to show commutativity of the diagrams
$$ \xymatrix{
S(X) \ar@{}[dr]|{\rm(I)} \ar[r]^-{\vartheta}
\ar@{=}[d] & Z\diamond S(X)   \ar[d]^{\overline{\ve}\diamond S(X)}\\
 S(X)  & I \diamond S(X) \ar[l]^-{\lambda_{S(X)}} \, ,} \qquad
\xymatrix{
S(X)\ar[d]_-{\vartheta} \ar@{}[dr]|{\rm(II)} \ar[r]^-{\vartheta} &
Z\diamond S(X) \ar[d]^{\overline{\delta}\diamond S(X)}\\
Z\diamond S(X) \ar[r]_-{Z\diamond \vartheta}& Z\diamond  Z\diamond S(X)\, ,}
$$
where $\overline{\ve}=\xi \cdot T(\ve)$ and $\overline{\delta}=\omega_{C, C}
\cdot T(\delta)$ are the counit
and the comultiplication for the $\V_\bT$-module-comonoid $\bZ=\mathcal{T}(C)$
 (see \ref{alg-comonoid}).
In the diagram
$$
\xymatrix{
S(X) \ar[drr]_{S(\lambda_X^{-1})} \ar[rr]^-{S(\theta)} &&
S(C \diamond X )\ar[d]^{S(\varepsilon\diamond X)} \ar[rr]^-{\chi_{C,X}} &&
T(C) \diamond S(X) \ar[d]^{T(\ve) \diamond S(X)}\\
&& S(I \diamond X)\ar[d]_{S(\lambda_X)} \ar[rr]_-{\chi_{I,X}} &&
T(I) \diamond S(X)\ar[d]^{\xi\diamond S(X)}\\
&&S(X) \ar[rr]_{\lambda_{S(X)}^{-1}} && I \diamond S(X),}
$$
\begin{itemize}
  \item the triangle commutes since $(X,\theta)\in {^C\A}$,
  \item the top rectangle commutes by naturality of $\chi$,
  \item the bottom rectangle commutes since $\bS$ is a $\bT$-comodule-monad
 (see diagram (3.6) in \cite{A});
\end{itemize}
it follows that diagram (I) is commutative. To show that (II) is also commutative,
consider the diagram
$$
\xymatrix{
S(X) \ar@{}[rrd]|{(1)}\ar[d]|{S(\theta)} \ar[rr]^-{S(\theta)} &&
S(C \diamond X ) \ar@{}[rrd]|{(2)} \ar[d]|{S(\delta \diamond X)} \ar[rr]^-{\chi_{C,X}} &&
 T(C) \diamond S(X) \ar[d]|{T(\delta) \diamond S(X)}\\
S(C \diamond X)\ar@{}[rrd]|{(3)}\ar[d]|{\chi_{C, X}} \ar[rr]_{S(C \diamond \theta)}&&
S(C \diamond C \diamond X)\ar@{}[rrd]|{(4)}\ar[d]|{\chi_{C, C \diamond X}}
\ar[rr]_-{\chi_{C \diamond C,X}}&& T(C \diamond C) \diamond S(X)
\ar[d]|{\omega_{C, C}\diamond S(X)}\\
T(C) \diamond S(X) \ar[rr]_-{T(C) \diamond S(\theta)}&&T(C) \diamond S(C \diamond X)
 \ar[rr]_{T(C) \diamond \chi_{C, X}} && T(C) \diamond T(C) \diamond S(X),}
$$ in which
\begin{itemize}
  \item rectangle (1) commutes since $(X,\theta)\in {^C\!\A}$,
  \item rectangle (2) and (3) commute by naturality of $\chi$,
  \item rectangle (4) commutes since $\bS$ is a $\bT$-comodule-monad
(see diagram (3.5) in \cite{A});
\end{itemize}
therefore the outer square (and hence (II)) is commutative. Thus,
$(X, \vartheta)\in {^{\bZ_0}\!\A}$, and since ${(S(X), (m_\bS)_X)\in \A_{\bS}}$,
in order to show that
$(S(X), (m_\bS)_X, \vartheta)\in {^\bZ\!(\A_{\bS}) }$,
we need commutativity of the diagram
$$
\xymatrix{
SS(X) \ar[d]|{SS(\theta)} \ar[rr]^-{(m_{\bS})_X} && S(X) \ar[d]|{S(\theta)}&&\\
SS(C \diamond X)\ar[d]|{S( \chi_{C, X})} \ar[rr]^-{(m_{\bS})_{C \diamond X}} &&
S(C \diamond X) \ar[rrd]^{\chi_{C, X}}&&\\
S(T(C) \diamond S(X)) \ar[rr]_-{\chi_{T(C), T(X)}}&&
T(T(C)) \diamond SS(X) \ar[rr]_-{(m_{\bT})_{C} \diamond (m_{\bS})_{C}} &&
T(C) \diamond S(X);}
$$
 since the rectangle commutes by naturality of $m_{\bS}$, while
the trapezoid commutes since $\bS$ is a $\bT$-comodule-monad
(see diagram (3.10) in \cite{A}) the outer paths commute, too.
\end{proof}

As an immediate consequence we obtain from \ref{com-fun}:

\begin{corollary}\label{coaction} For  $(X, \theta)\in {^C\!\A}$, the
$(X, \theta)$-component $\kappa_{(X, \theta)}: S(X) \to Z\diamond S(X)$ of the
$\bbG$-comodule structure on the composite
$\phi_{\bS}\, {^C\!U}:{^C\!\A} \to \A_{\bS}$ induced by the commutative diagram
{\rm(\ref{lemma.d})} is the composite
$$ S(X) \xrightarrow{S(\theta)} S(C\diamond X)
\xrightarrow{\chi_{C,X}}\mathcal{T}(C)\diamond S(X)=Z\diamond S(X).$$
\end{corollary}

Write $\mathcal{G}$ for the comonad on the category $\A_{\bS}$ generated by the adjunction
$$\phi_{\bS}\, {^C\!U} \dashv {^C\!\phi} \, U_{\bS}: \A_{\bS}\to {^C\!\A}.$$

\begin{proposition} For any $(X, h)\in \A_{\bS}$, the $(X, h)$-component of
the comonad morphism $t_K :\mathcal{G} \to \bbG$
induced by the commutative diagram {\rm(\ref{lemma.d})} is the composite
$$ \mathcal{G}(X,h)=S(C \diamond X) \xrightarrow{\chi_{C,X}} Z \diamond S(X)
 \xrightarrow{Z \diamond h} Z \diamond X=\bbG(X,h).$$
\end{proposition}

\begin{proof} Let $(X, h)\in \A_{\bS}$.
The $(X,h)$-component of the counit of the adjunction
$\phi_{\bS}\, {^C\! U} \dashv {^C\!\phi} \,U_{\bS}: \A_{\bS}\to {^C\!\A}$
is the composite
$$S(C \diamond X) \xrightarrow{S(\varepsilon \diamond X)} S(I \diamond X)
\xrightarrow{S(\lambda_X)} S(X) \xrightarrow{\,h} X;$$
it follows from \ref{com-fun} that the morphism
$$(t_K)_{(X,h)} :S(C \diamond X) =\mathcal{G}(X,h) \to \bbG(X,h)=Z\diamond X$$
 is the composite
$$ S(C \diamond X) \xrightarrow{\kappa_{_{^C\!\!\phi(X)}}}Z\diamond S(C \diamond X)
\xrightarrow{Z\diamond S(\varepsilon \diamond X)}Z\diamond S(I \diamond X)
\xrightarrow{Z\diamond S(\lambda_X)} Z\diamond S(X)
\xrightarrow{Z\diamond h}Z\diamond X.$$
From
$$^C\!\!\phi(X)=(C\diamond X,\, C \diamond X \xrightarrow{\delta \diamond X}
(C \otimes C) \diamond X \xrightarrow{\alpha_{C,C,X}} C \diamond (C \diamond X)),$$
we obtain by Corollary \ref{coaction} that $\kappa_{_{^C\!\phi(X)}}$ is the composite
$$S(C \diamond X) \xrightarrow{S(\delta \diamond X)} S((C \otimes C) \diamond X)
\xrightarrow{S(\alpha_{C,C,X})} S(C \diamond (C \diamond X))
\xrightarrow{\chi_{C,C \diamond X}} Z\diamond S(C \diamond X).$$
In the diagram
$$
\xymatrix{S(C \diamond X) \ar[rd]_{S(r^{-1}_C \diamond X)}\ar[r]^-{S(\delta \diamond X)}
 &S((C \otimes C) \diamond X)\ar[d]|{S((C \otimes \varepsilon) \diamond X)}
\ar[r]^-{S(\alpha_{C,C,X})} & S(C \diamond (C \diamond X))
\ar[d]|{S(C \diamond (\varepsilon \diamond X))}\ar[rr]^-{\chi_{C, C\diamond X}} &&
 Z \diamond S(C \diamond X)\ar[d]|{Z \diamond S(\varepsilon \diamond X)}\\
& S((C \otimes I) \diamond X) \ar[rd]_{S(r_C \diamond X)}\ar[r]_-{S(\alpha_{C,I,X})} &
S(C \diamond (I \diamond X)) \ar[d]|{S(C\diamond \lambda_X)}
\ar[rr]_-{\chi_{C, I\diamond X}} &&Z \diamond S(I \diamond X)
\ar[d]|{Z \diamond S(\lambda_X)}\\
&& S(C\diamond X) \ar[rr]_-{\chi_{C, X}} && Z \diamond S(X),}
$$
\begin{itemize}
  \item the three rectangles are commutative by naturality of $\alpha$ and $\chi$,
  \item the top triangle commutes since $\varepsilon$ is the counit for the coalgebra $C$,
  \item the bottom triangle commutes because $\diamond$ is a left action of $\V$ on $\A$.
\end{itemize}
Hence the outer paths commute and we have
$$\begin{array}{rl}
(t_K)_{(X,h)} & = \;(Z\diamond h)  (Z\diamond S(\lambda_X))
(Z\diamond S(\varepsilon \diamond X)) \kappa_{_{^C\!\!\phi(X)}}\\[+2mm]
&= \;(Z\diamond h)  (Z\diamond S(\lambda_X))
(Z\diamond S(\varepsilon \diamond X))\chi_{C,C \diamond X} S(\alpha_{C,C,X})
 (S(\delta \diamond X))\\[+2mm]
&= \;(Z\diamond h) \chi_{C, X} S(r_C \diamond X) S(r^{-1}_C \diamond X) =
(Z\diamond h)\chi_{C, X}. \end{array}$$
\end{proof}

Since for any $X \in \A$, $\phi_{\mathcal{S}}(X)=(S(X), (m_{\mathcal{S}})_X)$,
the following is immediate:

\begin{corollary}\label{free} For any $X \in \A$, the $\phi_{\mathcal{S}}(X)$-component
$(t_K)_{\phi_{\mathcal{S}}(X)}$ of the comonad morphism $t_K :\mathcal{G} \to \bbG$
is the composite
$$ S(C \diamond S(X)) \xrightarrow{\chi_{C,S(X)}} Z \diamond SS(X)
\xrightarrow{Z \diamond \,({m_{\mathcal{S}})_X}} Z \diamond S(X).$$
\end{corollary}

Combining this with Theorem \ref{comparison} and with Proposition \ref{galois}
and using $(t_K)_{\phi_{\mathcal{S}}(X)}=(t_K\phi_{\mathcal{S}})_X$ yields:

\begin{theorem}\label{main1} Under the assumptions {\rm\ref{ass.1}, \ref{ass.2}},
the functor
$K: {^C\!\A} \to {^\bZ\!(\A_{\bS})}$  in a
commutative diagram {\rm(\ref{lemma.d})}  is an equivalence of categories if
and only if
\begin{itemize}
  \item [(i)] the functor $\phi_{\bS} {^C \!U} : {^C\!\A} \to \A_\bS$ is comonadic and
  \item [(ii)] the composite
$$ S(C \diamond S(X)) \xr{\chi_{C,S(X)}} Z \diamond SS(X)
\xr{Z \diamond \, ({m_{\mathcal{S}})_X}} Z \diamond S(X)$$
is an isomorphism for all $X \in \A$, or, equivalently,
$\phi_{\bS} {^C \!U} : {^C\!\A} \to \A_\bS$ is a $\bbG$-Galois comodule functor.
\end{itemize}
\end{theorem}

In view of Proposition \ref{comonadicity}(i),  the preceding  theorem implies:

\begin{theorem}\label{split}  Assume that
$\A$ is Cauchy complete and that $e_\bS: I \to S$ is a split monomorphism.
Under the assumptions {\rm\ref{ass.1}, \ref{ass.2}},
 the functor $K: {^C\!\A} \to {^\bZ\!(\A_{\bS})}$ with commutative diagram
{\rm(\ref{lemma.d})} is an equivalence of categories
if and only if the functor $\phi_{\bS} {^C \!U} : {^C\!\A} \to \A_\bS$ is
$\bbG$-Galois.
\end{theorem}

We now obtain  the Fundamental Theorem of generalised Hopf modules
(see \cite[Theorem 5.3.1]{A}) as a particular case of Theorem \ref{main1}.

\begin{theorem}\label{aguiar} Under the assumptions {\rm\ref{ass.1}, \ref{ass.2}},
suppose that
\begin{itemize}
  \item [(i)]  $\A$ admits equalisers of coreflexive $\phi_\bS$-split pairs,
  \item [(ii)]  the functors $S$, $C \diamond -$, $C \diamond (C \diamond -):\A \to \A$ preserve these equalisers, and
  \item [(iii)] the functor $S$ is conservative.
\end{itemize}
Then the functor $K: {^C\!\A} \to {^\bZ\!(\A_{\bS})}$ in a commutative diagram
{\rm(\ref{lemma.d})} is an equivalence of categories
if and only if the functor $\phi_{\bS} {^C \!U} : {^C\!\A} \to \A_\bS$ is
$\bbG$-Galois.
\end{theorem}
\begin{proof} Since the functor $S$ is conservative and the category  $\A$ admits - and
the functor $S$ preserves - equalisers of coreflexive $\phi_\bS$-split pairs, it
follows from the dual of Beck's monadicity theorem (see \cite{Mc}) that the functor
$\phi_\bS :\A \to \A_\bS$
is comonadic, or equivalently, the monad $\bS$ is of effective descent type.
Since any $\phi_\bS$-split pair
is automatically $S$-split, we may apply Proposition \ref{comonadicity}(ii) to deduce that
the functor $\phi_{\bS} {^C\! U} : {^C\!\A} \to \A_\bS$ is comonadic. The result
now follows from Theorem \ref{main1}.
\end{proof}

\section{Bimonoids in duoidal categories}\label{bimod}

In \cite{AM}, Aguiar and Mahajan generalised bialgebras over fields to {\em bimonoids} in
  {\em  duoidal categories}, that is, categories with two monoidal structures  $\ast$
and  $\circ$.
Any object $A$ in such a category induces endofunctors $-\ast A$ and $-\circ A$
and for $A$ being a bimonoid
these functors have to be a comonad and a monad, respectively, related by a
mixed distributive law
(\cite[Definition 6.25]{AM}, compare \cite[Proposition 6.3]{MW-Bim}).
In \cite{BCZ}, B\"ohm, Chen and Zhang studied which structures are required to define
{\em Hopf monoids} in such categories. In this section we outline how their
Fundamental Theorem for  Hopf modules
(see \cite[Section 3.4]{BCZ}) can be seen as special case of the results
in the Sections \ref{sec-gen} and \ref{g-Hopf}.

Recall from \cite{AM} that duoidal  categories $\D$ are equipped with two
monoidal structures $(\D, \circ, I)$ and $(\D, \ast, J)$, along with a natural
transformation
$$\zeta_{W,X,Y,Z}:(W \circ X)\ast (Y \circ Z) \to (W \circ Y)\ast (X \circ Z),$$
 called
the \emph{interchange law}, and three morphisms
 $$\Delta: I \to I \ast I,\,\; \mu: J \circ J \to J, \,\; \tau: I \to J,$$
such that the conditions for \emph{associativity, unitality} and
\emph{compatibility of the units}   are satisfied.
For example, the compatibility of the units  means that the monoidal units $I$ and $J$
satisfy
\begin{itemize}
  \item  $(J, \mu, \tau)$ is a monoid in the monoidal category $(\D, \circ, I)$  and
  \item  $(I, \Delta, \tau)$ is a comonoid in the monoidal category $(\D, \ast, J)$.
\end{itemize}

It is pointed out in \cite{AM} that if $(\D, \circ, I, \ast, J)$ is a duoidal category
with interchange law
$\zeta$, then $(\D^{\text{op}},\ast,J,\circ,I)$ is also a duoidal category,
called the \emph{opposite}
duoidal category of $\D$. The interchange law
$\overline{\zeta}_{W,X,Y,Z}:(W \circ X) \ast (Y \circ Z) \to
(W \ast Y) \circ (X \ast Z)$ for this is given by the
$\D$-morphism $\zeta_{W,Y,X,Z}:(W \ast Y) \circ (X \ast Z)
\to (W \circ X) \ast (Y \circ Z)$.

\begin{thm}\label{bimonoid}{\bf Bimonoids.} \em A \emph{bimonoid} in a duoidal
category $\D$ is an object $A$
with a monoid structure $(A, m, e)$ in $(\D, \circ, I)$ and  a comonoid structure
$(A, \delta, \ve)$ in $(\D, \ast, J)$ inducing commutativity of the diagrams
$$
\xymatrix{
A \circ A \ar@{}[drr]|{\rm(I)}\ar[r]^-{m} \ar[d]_-{\delta \circ \delta}&
A \ar[r]^-{\delta}& A \circ A \\
(A \ast A)\circ (A \ast A) \ar[rr]_-{\zeta}&& (A \circ A)\ast (A \circ A)
\ar[u]_-{m\ast m},}
$$
$$
 \xymatrix{
A\circ A \ar@{}[dr]|{\rm(II)} \ar[r]^-{\ve \circ \ve}
\ar[d]_m & J\circ J  \ar[d]^\mu \\
 A \ar[r]_\ve & J , } \qquad
\xymatrix{
I \ar@{}[dr]|{\rm(III)} \ar[r]^-{e} \ar[d]_-{\Delta} & A \ar[d]^-{\delta}\\
I \ast I \ar[r]_-{e\ast e}& A \ast A,}
 \qquad
\xymatrix{ I \ar[r]^e \ar[dr]_\tau & A \ar@{}[dl]|(.35){\rm(IV)} \ar[d]^\ve\\
         & J.}$$

A morphism of bimonoids is a morphism of the underlying monoids and comonoids.
\end{thm}
Recall \cite[Proposition 6.27]{AM} that the tensor units I and J carry a unique
bimonoid structure and that
the morphism $\tau: I \to J$ is a morphism of bimonoids.

Fix a duoidal  category $(\D, \circ, I, \ast, J)$ and a bimonoid
$(A, m, e, \delta, \ve)$ in $\D$. Since
$(A, A\circ A \xr{m} A, e:I \xr{e} A)$ is a monoid in the monoidal category
$(\D, \circ, I)$, while $(A,A \xr{\delta} A \ast A, A \xr{\varepsilon} J)$
is a comonoid in the monoidal category $(\D, \ast, J)$, one has by
\ref{action} the monad $\bT^r_A$ and the comonad $\bG^r_A$ on $\D$.
Recall that the functor part of the monad $\bT^r_A$ (resp. comonad $\bG^r_A$)
is the functor $-\circ A : \D \to \D$
(resp. $-\ast A : \D \to \D$).

It is shown in \cite{BS} that $\bT^r_A $ is an opmonoidal monad on the monoidal
category $(\D, \ast, J)$, with
the structure morphisms
$$\overline{\omega}_{X,Y}:(X \ast Y)\circ A \xr{(X \ast Y)\circ \delta}
(X \ast Y)\circ (Y \ast A)\xr{\zeta} (X \circ A)\ast (Y \circ A), $$
  $$\overline{\xi}:J \circ A \xr{J \circ \ve} J \circ I \xr{\simeq} J.$$
It follows that
the category $\D_A=\D_{\bT^r_A}$ is monoidal. Note that $((A, m),\delta, \ve)$
is an object of $\textbf{\text{Mon}}(\D_A)$: Clearly
$(A, m)\in \D_A$ and the comultiplication $\delta$ and the counit $\ve$ of $A$ are
morphisms of right $A$-comodules by the diagrams
(I) and (II) in \ref{bimonoid}, respectively. Hence $((A,m),\delta,\varepsilon)$
is a $\bT^r_A$-module-comonoid.

Thus, we have
\begin{itemize}
  \item  the opmonoidal monad $\bT^r_A$  on the monoidal category
  $(\D, \ast, J)$,
  \item the left $(\D, \ast, J)$-category structure on $\D$ given
by $X \diamond Y= X \ast Y$,
  \item the $\bT^r_A$-comodule-monad $(\bT^r_A,\, \overline{\omega})$ on $\A$, and
  \item the $\bT^r_A$-module-comonoid $((A,m),\delta,\varepsilon)$.
\end{itemize}

Hence, we may apply the results of Section 2 to the present situation.
In particular, Proposition \ref{entwining} gives (see also \cite[Section 2]{BCZ}):

\begin{proposition}\label{entwining1} The natural transformation
$$\lambda: (-\ast A) \circ A \xr{(-\ast A) \circ \delta} (-\ast A) \circ (A\ast A)
 \xr{\zeta} (-\circ A) \ast (A \circ A)
\xr{(-\circ A) \ast m} (-\circ A)\ast A$$
is a mixed distributive law from the monad $\bT^r_A$ to the comonad $\bG^r_A$.
\end{proposition}

We write $\D^A_A$ for the category
$(\D_{\bT^r_A})^{\widehat{\bG^r_A}}=(\D_A)^{\widehat{\bG^r_A}}$,
where ${\widehat{\bG^r_A}}$
is the lifting of $\bG^r_A$ to $\D_{\bT^r_A}=\D_A$ corresponding to the
mixed distributive law $\lambda$.
This is called the category of  $A$-\emph{Hopf modules} in \cite[Section 3]{BCZ}.
Thus, an $A$-Hopf module is a triple $(X,h,\vartheta)$ such that
$(X,\,h:X \circ A \to A)$ is a right $A$-module in the monoidal category $(\D, \circ, I)$,
$(X, \,\vartheta: X \to X \ast A)$ is a right $A$-comodule in the monoidal category
 $(\D, \ast, J)$
with commutative diagram
$$
\xymatrix{X\circ A\ar[d]_{\vartheta \circ A} \ar[r]^{h} & X \ar[r]^{\vartheta\quad} &
 X \ast A \\
(X \ast A)\circ A  \ar[rr]_-{\lambda_X}&& (X \circ A)\ast A \ar[u]_{h \ast A}.}
$$
Since $(I, \Delta, \tau)$ is a comonoid in the monoidal category $(\D, \ast, J)$,
one has the category $\D^I$ of $I$-comodules on this monoidal category.
Recall that $\D^I$ is the Eilenberg-Moore category of $\bG^r_I$-comodules.
Now Lemma \ref{lemma1} implies:

\begin{thm}\label{comparison1}{\bf Comparison functor $-\circ A: \D^I \to \D^A_{A}$.} \em
 The assignment
$$(X, \theta) \rightsquigarrow (X \circ A, \overline{m},\overline{\vartheta}),$$
where $\overline{m}:(X \circ A)\circ A \to X \circ A$
is the composite $$(X \circ A)\circ A \xr{a_{X,A,A}} X \circ (A\circ A) \xr{X \circ \, m} X \circ A,$$
while
$\overline{\vartheta}: X \circ A \to (X\circ A)\ast A$ is the composite
$$X\circ A \xr{\vartheta \circ A} (X \ast I)\circ A \xr{ (X \ast I)\circ \delta}
(X \ast I) \circ (A \ast A)
\xr{\zeta} (X\circ A)\ast (I \circ A)\simeq (X\circ A)\ast A,$$
yields a comparison functor
 $K = -\circ A: \D^I \to \D^A_{A}$
with commutative  diagram
$$
\xymatrix{ \D^I \ar[rr]^-{K} \ar[d]_{U^I} && \D^A_{\!A} \ar[d]^{U^{\widehat{\bG^r_A}}}\\
\D \ar[rr]^{\phi_{A}}&& \D_{\!A}.}
$$
\end{thm}

Write $\mathcal{G}$ for the comonad on the category $\D_{A}$ generated by the adjunction
$$\phi_{A} {U^I} \dashv \phi^I U_{A}: \D_{A}\to {\D^I}.$$

\begin{proposition}\label{coaction1} For any $(X,h)\in \D_A$, the $(X,h)$-component
of the comonad morphism $t_K :\mathcal{G} \to \widehat{\bG^r_{\!A}}$
induced by the diagram in {\rm \ref{comparison1}}, is the composite
$$ (X \ast I)\circ A \xr{ (X \ast I)\circ \delta} (X \ast I) \circ (A \ast A)
\xr{\zeta} (X\circ A)\ast (I \circ A)\simeq (X\circ A)\ast A \xr{h \ast A}X \ast A.$$

For any $X\in \D$,
$\phi_A(X)=(X \circ A,\, (X \circ \, m)\cdot a_{X,A,A}:(X\circ A)\circ A\to X\circ A),$
and the  $\phi_A(X)$-component $(t_K)_{\phi_A(X)}$ of $t_K$ is the composite
$$
\xymatrix@C=40pt @R=30pt{
((X \circ A)\ast I)\circ A \ar@{..>}[d]|{(t_K)_{\phi_A(X)}}
\ar[r]^-{((X\circ A) \ast I)\circ \delta} & ((X\circ A) \ast I) \circ (A \ast A)
\ar[r]^-{\zeta} & ((X\circ A)\circ A)\ast (I \circ A) \ar[d]^\simeq \\
(X \circ A) \ast A &(X\circ (A\circ A))\ast A \ar[l]^-{(X \circ\, m) \ast A} &
((X\circ A)\circ A)\ast A \ar[l]^\simeq \, .}
$$
\end{proposition}

Applying now Theorem \ref{main1} yields:

\begin{theorem}\label{comparison2} Let $(\D,\circ,I,\ast,J)$ be a duoidal category and
$(A, m, e, \delta, \ve)$ a bimonoid  in $\D$. Then
the comparison functor $K: \D^I \to \D_{\!A}^A$ is an equivalence of categories if
and only if
\begin{itemize}
  \item [(i)]  the functor $\phi_A U^I : \D^I \to \D_A$ is comonadic  and
  \item [(ii)] the morphism ${(t_K)_{\phi_A(X)}}$ (in {\rm\ref{coaction1}})
        is an isomorphism for any $X \in \D$, or, equivalently,
        $\phi_A U^I : \D^I \to \D_A$ is a $\bG^r_A$-Galois comodule functor.
\end{itemize}
\end{theorem}

\begin{theorem}\label{unit.split}Let $(\D, \circ, I, \ast, J)$ be a duoidal
category with Cauchy complete $\D$.
If the morphism $\tau: I \to J$ is a split monomorphism, then for any bimonoid
$(A, m, e, \delta, \ve)$ in $\D$, the  functor
${K: \D^I \to \D_{\!A}^A}$ (in {\rm \ref{comparison1}}) is an equivalence of categories if
and only if $\phi_A U^I : \D^I \to \D_A$ is $\bG^r_A$-Galois.
\end{theorem}
\begin{proof}Since $\ve \cdot e=\tau $ (see Diagram (IV) in \ref{bimonoid}) and
since $\tau $ is a split monomorphism by hypothesis, so also is the morphism $e:I \to A$.
It then follows that the unit of the monad $\bT^r_A=- \circ A$ is a split monomorphism
and Theorem \ref{split} completes the proof.
\end{proof}

The following elementary observation is of use for our investigations.

\begin{proposition}\label{split1} Let $(\D, \circ, I, \ast, J)$ be a duoidal category with Cauchy complete $\D$.
If the unit of the adjunction
$$\xymatrix{
\D^I \ar@/^.6pc/[r]^-{U^I} \ar@{}[r]|-{\perp} &\D \ar@{}[r]|-{\perp}
\ar@/^.6pc/[r]^-{\phi_J} \ar@/^.6pc/[l]^-{\phi^I} &\D_J \ar@/^.6pc/[l]^-{U_J}}
$$
is a split monomorphism, then for any bimonoid $(A, m, e, \delta, \ve)$ in $\D$,
the functor $\D^I \xr{U^I}\D \xr{\phi_A}\D_A$ is comonadic.
\end{proposition}
\begin{proof} Note first that the commutativity of the diagrams (II) and (IV)
in \ref{bimonoid} allow to consider
 $\ve: A \to J$ as a morphism from the monoid $(A, m, e)$ to the monoid $(J, \mu, \tau)$
in the monoidal category $(\D, \circ, I)$. Then the composites
$A \circ J \xr{\ve \circ J} J \circ J \xr{\mu} J $ and
$J \circ A \xr{J \circ\, \ve} J \circ J \xr{\mu}J$ give the structure of an
$(A,A)$-bimodule on $J$, and so,
by Remark \ref{remark}, the triangle in the diagram
$$
\xymatrix{\D^I \ar@{=}[d]\ar[r]^-{U^I} &\D \ar[dr]|{\phi_J=J \otimes -}
\ar[r]^-{\overline{\phi}_A} &
\widetilde\D_\bbT\ar[d]^{J \otimes_A-}\\
\D^I \ar[r]_-{U^I} &\D \ar[r]_-{\phi_J} &\D\!_J}
$$ is commutative, implying that the outer diagram is also commutative.
Since $\D$ is assumed to be Cauchy
complete, so also is $\D^I$. Now apply Proposition \ref{split.comon} to conclude that
the composite $\overline{\phi}_A U^I$ is conservative, and that any coreflexive
$(\overline{\phi}_A U^I)$-split pair of morphisms has a split equaliser in $\D^I$.

Next, since $\iota\, \overline{\phi}_A=\phi_A$, where $\iota:\widetilde\D_\bbT \to \D_A$
is the canonical embedding, the composite $\overline{\phi}_A U^I$ is conservative
if and only if $\iota\,\overline{\phi}_A U^I=\phi_A U^I$
is conservative, and a pair of morphisms in $D^I$ is $\overline{\phi}_A U^I$-split
if and only if it is $\phi_A U^I$-split. Thus,
$\overline{\phi}_A U^I$ is conservative and any $\overline{\phi}_A U^I$-split pair
of morphisms
has a split equaliser in $\D^I$. It follows -- since the composite $\phi_A U^I$ admits
as a right adjoint the composite $\phi^I U_A$ -- that $\overline{\phi}_A U^I$ is comonadic.
\end{proof}

Combining Propositions \ref{comparison2} and \ref{split1} we obtain:

\begin{theorem}\label{main2} In the situation of Proposition {\rm\ref{split1}},
the  functor
${K: \D^I \to \D_{\!A}^A}$ (in {\rm \ref{comparison1}}) is an equivalence of categories if
and only if $\phi_A U^I : \D^I \to \D_A$ is $\bG^r_A$-Galois.
\end{theorem}

Since a left adjoint functor is full and faithful if and only if the unit of the
adjunction is an isomorphism (hence a split monomorphism), it follows immediately:

\begin{corollary}\label{f.and.f} Let $(\D, \circ, I, \ast, J)$ be a duoidal category
with Cauchy complete $\D$.
If the composite $\phi_J U^I: \D^I \to \D_J$ is full and faithful,
then, for any bimonoid $(A, m, e, \delta, \ve)$ in $\D$,
the functor ${K: \D^I \to \D_{A}^A}$ (in {\rm \ref{comparison1}}) is an equivalence of
categories if and only if $\phi_A U^I : \D^I \to \D_A$ is $\bG^r_A$-Galois.
\end{corollary}

Considering any bimonoid $(A, m, e, \delta, \ve)$ in $\D$ as a bimonoid in
$\D^{\text{op}}$ (see \cite[Remark 6.26]{AM}) and applying
the duality explained in \cite{AM}, the Theorems \ref{comparison2},
\ref{main2} and Corollary \ref{f.and.f} yield the following:

\begin{theorem}\label{d.comparison2} Let $(\D, \circ,I,\ast,J)$ be a duoidal category and
$(A, m, e, \delta, \ve)$ a bimonoid  in $\D$. Then
the comparison functor
$$K'= -\ast A:\D_J \to \D_{\!A}^A$$
 is an equivalence of categories if and only if
\begin{itemize}
  \item [(i)]  the functor $\phi^A U_J : \D_J \to \D^A$ is monadic and
  \item [(ii)] $\phi^A U_J : \D_J \to \D^A$ is a $\bT^r_{A}$-Galois module functor, or,
equivalently, the following composite is an isomorphism for all $X \in \D$:
$$
\xymatrix@C=40pt @R=30pt{
(X \ast A) \circ A \ar@{..>}[d]|{(t_{K'})_{\phi^A(X)}} \ar[r]^-{(X \ast \delta) \circ A}&
 (X\ast (A\ast A))\circ A \ar[r]^\simeq & ((X\ast A)\ast A)\circ A   \ar[d]^\simeq \\
((X \ast A)\circ J)\ast A  & ((X\ast A) \circ J) \ast (A \circ A)
 \ar[l]_-{ ((X \ast A) \circ J)\ast m}
& ((X\ast A)\ast A)\circ (J \ast A)  \ar[l]_-{\zeta}.}
$$
\end{itemize}
\end{theorem}

Cauchy completeness of $\D$ allows for the following characterisations.

\begin{theorem}\label{unit.split.d} Let $(\D, \circ, I, \ast, J)$ be a duoidal
category with Cauchy complete $\D$.
Assume either of the conditions
\begin{rlist}
\item  the morphism $\tau: I \to J$ is a split epimorphism, or
\item the unit of the adjunction
$\xymatrix{
\D_J \ar@/_.6pc/[r]_-{U_J} \ar@{}[r]|-{\perp} &\D \ar@{}[r]|-{\perp} \ar@/_.6pc/[r]_-{\phi^I} \ar@/_.6pc/[l]_-{\phi_J} &\D^I \ar@/_.6pc/[l]_-{U^I}}
$ is a split epimorphism, or
\item the composite $\phi^I U_J: \D_J \to \D^I$ is full and faithful.
\end{rlist}
 Then, for any bimonoid $(A, m, e, \delta, \ve)$ in $\D$, the  functor
${K'=-\ast A: \D_J \to \D_{A}^A}$
 is an equivalence of categories if
and only if $\phi^A U_J : \D_J \to \D^A$ is $\bT^r_{A}$-Galois.
\end{theorem}

Note that Corollary \ref{f.and.f} subsumes \cite[Theorem 3.11]{BCZ}, while
 \cite[Theorem 3.13]{BCZ} is a consequence of  Theorem \ref{unit.split.d} (iii).

\bigskip

{\bf Acknowledgments.} This research was partially supported by
Volkswagen Foundation (Ref.: I/85989).
The first author also gratefully acknowledges the support of the
Shota Rustaveli National Science Foundation Grant DI/12/5-103/11.

\bigskip

\noindent
{\bf Addresses:} \\[+1mm]
{A. Razmadze Mathematical Institute of I. Javakhishvili Tbilisi State University,\\
6, Tamarashvili Str.,  {\small and} \\
 {Tbilisi Centre for Mathematical Sciences,
Chavchavadze Ave. 75, 3/35, \\
Tbilisi 0177},  Georgia.\\
    {\small bachi@rmi.ge}\\[+1mm]
{Department of Mathematics of HHU, 40225 D\"usseldorf, Germany},\\
  {\small wisbauer@math.uni-duesseldorf.de}


\begin{thebibliography}{11}

\bibitem{A} Aguiar, M. and Chase, S.U., {\em Generalized Hopf modules for bimonads,}
 Theory Appl. Categ. \textbf{27} (2013),  263--326.

\bibitem{AM} Aguiar, M. and Mahajan, S.,
{\em Monoidal Functors, Species and Hopf Algebras,}
CRM Monograph Series \textbf{29}, Amer. Math. Soc. Providence (2010).

\bibitem{B} Barr, M., {\em The Point of Empty Set}, Cahiers
Topologie G\'eom. Diff\'erentielle Cat\'egoriques \textbf{13}
(1972), 357--368.

\bibitem{Bn} B\'{e}nabou, J., {\em Introduction to bicategories,}
Lecture Notes in Mathematics Vol. \textbf{40} (1967), 1--77.

\bibitem{BBW} B\"ohm, G., Brzezi\'nski, T. and Wisbauer, R.,
    {\em Monads and comonads on module categories},
  J. Algebra \textbf{322} (2009), 1719--1747.

\bibitem{BCZ} B\"{o}hm, G., Chen, Y.  and Zhang, L.,
 {\em On Hopf monoids in duoidal categories,}
arXiv:1212.1018v1 (2012).

\bibitem{BS} Booker, T. and  and Street, R.,
   {\em Tannaka duality and convolution for duoidal categories,}
arXiv.org/abs/1111.5659v1 (2011).

\bibitem{Bo} Borceux, F., {\em Handbook of Categorical Algebra},
vol.  \textbf{2}, Cambridge University Press, Cambridge (1994).

\bibitem{BV} Brugui\`eres, A. and Virelizier, A., {\em Hopf monads},
   Adv. Math. \textbf{215(2)} (2007), 679--733.

\bibitem{BLV}   Brugui\`{e}res, A., Lack, S. and Virelizier, A.,
    {\em Hopf monads on monoidal categories},
    Adv. Math. \textbf{227(2)} (2011), 745--800.

\bibitem{McC} McCrudden, P., {\em Opmonoidal monads},
     Theory Appl. Categ. \textbf{10} (2002), 469--485.

\bibitem{Mc} Mac Lane, S., {\em Categories for the Working Mathematician},
2nd edn, Springer-Verlag, New York (1998).

\bibitem{M}    Mesablishvili, B.,  {\em Entwining Structures in Monoidal
Categories},  J. Algebra  \textbf{319} (2008), 2496--2517.

\bibitem{Me} Mesablishvili, B., {\em Monads of effective descent type and comonadicity,}
Theory Appl. Categ.  \textbf{16} (2006),  1--45.

\bibitem{MW-Bim}  Mesablishvili, B. and Wisbauer, R.,
{\em Bimonads and Hopf monads on categories},
 J. K-Theory 7(2) (2011), 349-388.

\bibitem{MW-Gal}    Mesablishvili, B.  and Wisbauer, R.,
{\em Galois functors and entwining structures},
  J. Algebra  \textbf{324} (2010), 464-506.

\bibitem{MW-Op}    Mesablishvili, B.  and Wisbauer, R.,
{\em Notes on bimonads and Hopf monads},
Theory Appl. Categ.  \textbf{26} (2012),  281--303.

\bibitem{Pa} Pareigis, B., {\em Non-additive ring and module theory I. General theory of monoids, }
Publ. Math. Debrecen \textbf{24} (1977), 189--204.

\bibitem{W}  Wisbauer, R.,  {\em Algebras versus coalgebras},
    Appl. Categ. Struct. \textbf{16}(1-2)  (2008), 255--295.

\bibitem{Wf}   Wolff, H., {\em V-Localizations and V-monads}.
 J. Algebra  \textbf{24} (1973), 405--438.

\end{thebibliography}
\end{document}